\documentclass[12pt]{article}
\usepackage{amsmath}
\usepackage{amssymb}
\usepackage{amsthm}
\usepackage{epsfig} 
\usepackage{wrapfig}
\usepackage{pdfpages} 
\usepackage{xcolor}
\usepackage[small]{caption}
\usepackage{subcaption}
\usepackage[export]{adjustbox}
\usepackage[pdftex,colorlinks]{hyperref}

\newcommand {\R}{\mathbb R}

\newcommand{\e}{\mathbb E}

\theoremstyle{definition}\newtheorem{thm}{Theorem}
\theoremstyle{definition}\newtheorem{lem}[thm]{Lemma}
\theoremstyle{definition}\newtheorem{cor}[thm]{Corollary}
\theoremstyle{definition}
\theoremstyle{definition}\newtheorem{defi}[thm]{Definition}
\theoremstyle{definition}\newtheorem{rem}[thm]{Remark}
\theoremstyle{definition}\newtheorem{prop}[thm]{Proposition}
\theoremstyle{definition}
\theoremstyle{definition}
\numberwithin{equation}{section}

\newcommand{\PP}{\mathbb P}

\newcommand{\Z}{\mathbb Z}

\newcommand{\ind}{\mathbf 1}

\newcommand{\bbN}{\Bbb{N}}

\newcommand{\E}{\e}

\newcommand{\calG}{{\cal G}}

\newcommand{\iid}{\hbox{i.i.d.}}

\begin{document}

\author{Fran\c{c}ois Baccelli\thanks{baccelli@math.utexas.edu} \\{\small The University of Texas at Austin} 
\and Antonio Sodre\thanks{asodre@math.utexas.edu}\\{\small The University of Texas at Austin}}
\title{Renewal Population Dynamics and their Eternal Family Trees}
\date{}
\maketitle
\begin{abstract}
Based on a simple object, an $\hbox{i.i.d.}$ sequence of positive integer-valued random variables, $\{a_n\}_{n\in \Z}$, we introduce and study two random structures and their connections. First, a population dynamics, in which each individual is born at time $n$ and dies at time $n+a_n$. This dynamics is that of a D/GI/$\infty$ queue, with arrivals at integer times and service times given by $\{a_n\}_{n\in \Z}$. Second, the directed random graph $T^f$ on $\Z$ generated by the random map $f(n)=n+a_n$.  Only assuming $\E[a_0]<\infty$ and $\PP[a_0=1]>0$, we show that, in steady state, the population dynamics is regenerative, with one individual alive at each regenerative epochs. We identify a unimodular structure in this dynamics.
More precisely, $T^f$ is a unimodular directed tree, in which $f(n)$ is the parent of $n$. This tree has a unique bi-infinite path.  Moreover, $T^f$  splits the integers into two categories: ephemeral integers, with a finite number of descendants of all degrees, and successful integers, with an infinite number. Each regenerative epoch is a successful individual such that all integers less than it are its descendants of some order. Ephemeral, successful, and regenerative integers form stationary and mixing point processes on $\Z$.  

 \end{abstract}


\section*{Introduction}

To each integer $n$ we assign a positive random integer
$a_n$. Then, $n$ is mapped $a_n$ units to the right. Given a probability space $(\Omega,\mathcal{F},\PP)$ supporting the random sequence $\{a_n\}_{n\in \Z}$, we consider the function $f: \Omega\times \Z\to \Z$ defined by $f(\omega,n)=n+a_n(\omega)$. We assume the sequence $\{a_n\}_{n\in \Z}$ is $\hbox{i.i.d.}$ We study two objects arising from the random map $f$: a population dynamics
and a directed random graph. When there is no chance of ambiguity, we omit $\omega$ in our notation. 

In the population dynamics, 
individual $n$ is {\em born} at time $n$ and {\em dies} at time $f(n)$.
More precisely, the lifespan of individual $n$ is assumed to be $[n,n+a_n)$.
Let $\textbf{N}:=\{N_n\}_{n\in \Z}$ be the discrete time random process of the number of individuals
alive at time $n$, or simply the {\em population process}.
The process $\textbf{N}$ can be seen as the number of customers at arrival epochs in a D/GI/$\infty$
queue, namely, a queuing system with one arrival at every integer time, and $\hbox{i.i.d.}$ integer-valued service times, all distributed like $a_0$. As a new arrival takes place at each integer, the population
never goes extinct or, equivalently, the queue is never empty. Nonetheless, as we shall see, when $\E[a_0]<\infty$, the stationary population process is regenerative. To visualize the number of individuals
alive at time 0, count the edges that cross over 0, {\em and} count the
individual born at time 0 (Figure 1).

\begin{figure}
\center
      \includegraphics[width=1\textwidth]{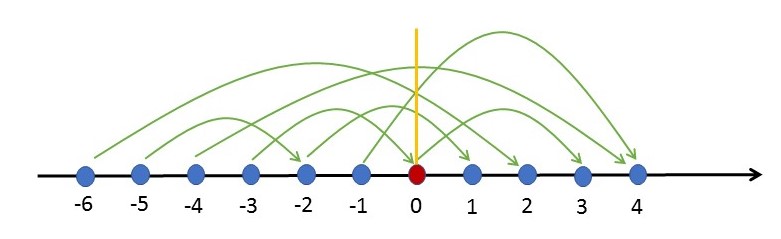}
  \caption{\textbf{The population process.} Here, $a_{-6}=10,~a_{-5}=3,a_{-4}=8,~a_{-3}=3,~a_{-2}=3,~\hbox{and}~a_{-1}=6$. Assuming $a_{-i}<i$ for all $i>6$, there are four edges crossing $0$. Individuals  -6,-4,-2, and -1 are still alive at time 0. Hence, $N_0=5$, as it includes the individual born at 0.}
\end{figure}

By letting $V^f=\Z$ and $E^f=\{(n,f(n)):n\in \Z\}$,
the random map $f$ also induces a random directed graph $T^f$. Further assuming $\PP[a_0=1]>0$, we will show $T^f$ is a directed tree.

We interpret $T^f$ as a family tree and connect it to the population process $\textbf{N}$.
In order to gain insight, we draw a parallel with the classical age-dependent branching process. Such a process is
built upon a lifetime distribution $L$ on $\mathbb R^+$ and an offspring distribution $O$ on $\mathbb N$.
In this classical model, the first individual is born at time 0 and has a lifetime $l$ sampled from $L$.
When the first individual dies at $l$, it is replaced by its offspring, whose cardinality is sampled from $O$.
From then on, each individual statistically behaves as the first one, with all individuals having independent
lifetimes and offspring cardinalities (see, e.g., \cite{waugh1955age}). 

Our family tree, in which individuals are indexed by $\Z$, is obtained by declaring that the individual
$f(n)$ is the parent of $n$, $f^2(n)$ is its grandparent, and so on. In terms of interpretation, this requires that we look at the  
population dynamics in reverse time. In ``reverse time'', individual $n$ ``dies'' at time $n$
(note that it is the only one to die at that time) and was ``born earlier'', namely, at time $n+a_n$. 
Since individual $m$ is born at time $n$ if 
$f(n)=m$, the set of children of $n$ is $f^{-1}(n)$, the
set of its grandchildren is $f^{-2}(n)$, and so on.  As in the age-dependent branching process, each individual has exactly one parent, but may have no children,
and the death time of an individual coincides with the birth time of its children.
Also notice $f^{-1}(n)\subset (-\infty,n-1]$. That is, as in the natural enumerations
used in branching processes, the children of individual $n$ have ``larger'' indices than that of individual $n$
(recall that individuals are enumerated using their ``death times''). Hence, each individual is born at the death of its parent and dies ``after'' its parent. Figure 2 illustrates the relation between the population process and $T^f$.
\begin{figure}
\center
      \includegraphics[width=1\textwidth]{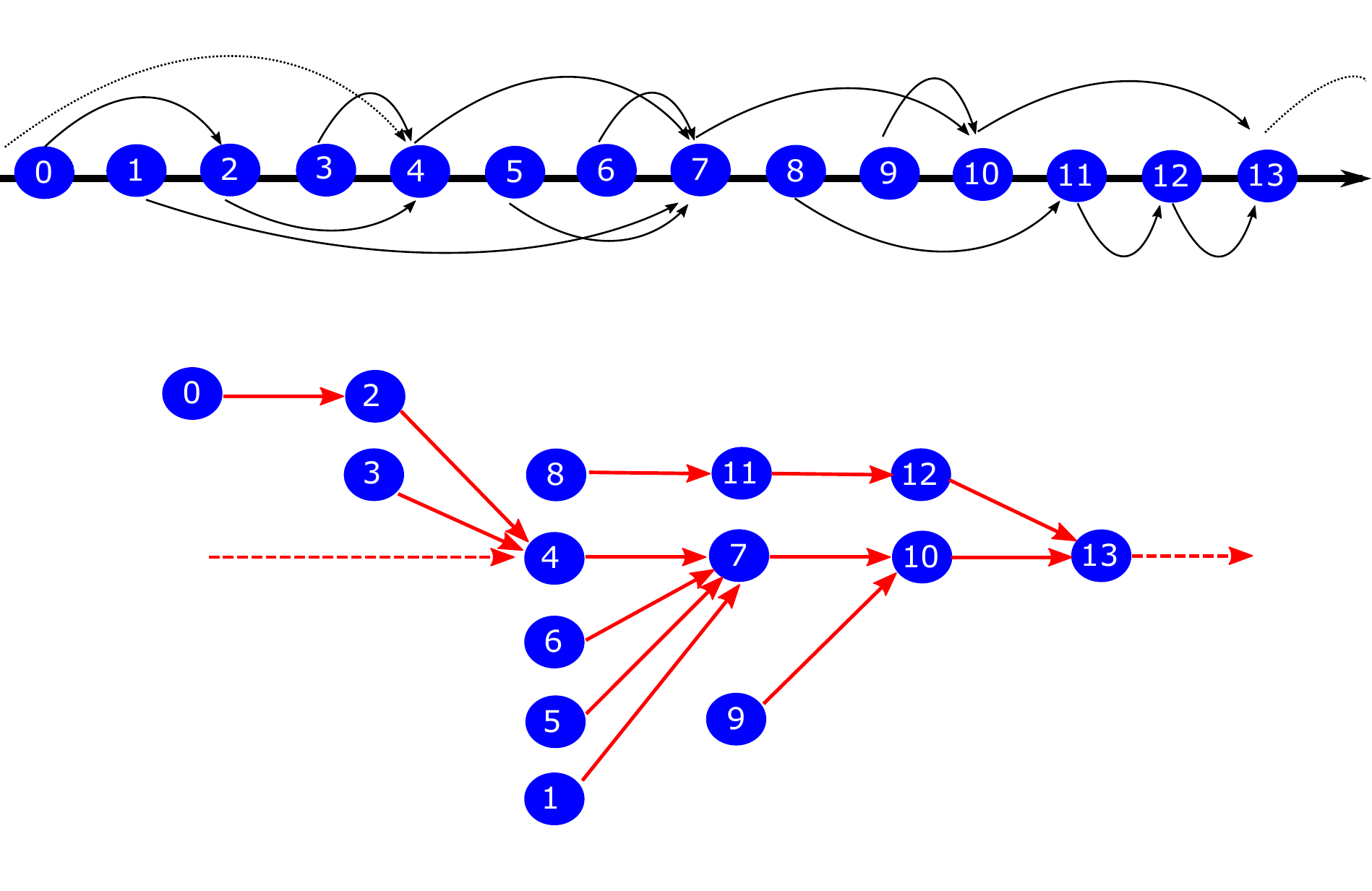}
  \caption{\textbf{From the population process to the family tree.} The top figure depicts the $f$ dynamics running on $\Z$. The curved black lines represent individual lifespans. For example, individual $7$ lives three units of time. The bottom figure depicts the tree representation of the dynamics on $\Z$, with the red edges representing parenthood. There, $7$ is the child of 10, and has four children: 1,4, 5, and 6. Individual $8$, for example, has no children.}
\end{figure}

However, our age-dependent family tree is far from being that of the age-dependent branching process discussed above.
In particular, there is no first individual: our family tree is {\em eternal} \cite{BOAunimodular}. More importantly, it lacks the independence
properties of branching processes. In particular, the offspring cardinalities of different individuals are dependent.  

In the age-dependent branching process described above, if we set $L=1$,
we recover the Bienaym\'e-Galton-Watson process. Despite the fact that the building
block of both our model and the Bienaym\'e-Galton-Watson model is just a sequence of
$\hbox{i.i.d.}$ random variables, the two models are quite different.
In the former, the $\hbox{i.i.d.}$ random variables define the offspring cardinalities. In the latter, 
they define the lifespans of individuals. Moreover, as we shall see, our model is
always {\em critical}: the mean offspring cardinality is one for all individuals. 

The fact that, when $\PP[a_0=1]>0$, $T^f$ is a unimodular directed tree \cite{BOAunimodular} allow us
 to complement the classical queuing and regenerative analysis
by structural observations that we believe to be new:
\begin{enumerate}
\item this tree is two ended;
\item an integer (individual) is either successful or ephemeral depending on whether the number of
its descendants (pre-images by $f$) of all orders is infinite or finite a.s.;
\item the set of successful (resp. ephemeral) integers forms a stationary point process;
\item there is a stationary thinning of the successful point process consisting of individual
such that all individuals born after one of them are its descendants - these are
called original ancestors;
\item each individual has a finite number of cousins of all orders.
\end{enumerate}
These structural observations are completed by closed form expressions for the
intensities of the point processes in question.

The last interpretation of the family tree pertains to renewal theory.
One can see $n,f(n),f^2(n),\cdots$ as a renewal process on the integers with interarrivals distributed like $a_0$ and starting from time $n$. The graph $T^f=(V^f,E^f)$ can hence be seen as the {\em mesh} of all
such renewal processes.
By mesh, we mean that the renewal process of $n$ merges with that of $m$ at the first time when
the orbits $\{f^p(n)\}_{n\in \bbN}$ and $\{f^q(m)\}_{m\in \bbN}$ meet.

\section{Population and queue dynamics} 
\label{populationdynamics}

In this section, we study the population process 
\begin{align}
\label{individualsalive}
N_n=\#\{\hbox{all $m\in \Z$~such that } m<n~\hbox{and}~f(m)>n\}+1,\quad n\in \mathbb Z.
\end{align}
Section \ref{generationgeneralcase} introduces definitions and notation. In Section \ref{generalcasemainresults}, we show the population process is regenerative with independent cycles. In Section \ref{analyticsofN}, an explicit formula for the moment generating function of $N_0$ is given and it is indicated this random variable is always light-tailed, regardless of the distribution of $a_0$. Finally, in Section \ref{geometricmarks}, we work out the case in which $a_0$ follows a geometric distribution and, consequently, the population process is Markovian. The notions and proof techniques in this section are classical and details are kept to a minimum.

\subsection{Definitions and assumptions}
\label{generationgeneralcase} 

Let $(\Omega,\mathcal{F}, \PP,\{\theta_n\}_{n\in\Z})$ be the underlying probability space supporting all random elements discussed in this paper, endowed with discrete flow. We assume $\PP$ is preserved by $\theta_n$, i.e., for all $n$, 
\begin{align}
\label{stationaryPP}
\PP\circ \theta_n^{-1}=\PP. 
\end{align}
A random integer-valued discrete sequence $\textbf{W}=\{W_n\}_{n\in \Z}$ defined on $\Omega$ is compatible with the flow $\{\theta_n\}_{n\in \Z}$ if $W_n(\omega)=W_0(\theta_n \omega)$ for all $n\in \Z$. Notice that, given (\ref{stationaryPP}), if a process $\textbf{W}$ is compatible with $\{\theta_n\}_{n\in \Z}$, then it is strictly stationary. All integer-valued discrete sequences considered here are $\{\theta_n\}_{n\in \Z}-$compatible (or, for short, stationary).  

In particular, since $\textbf{a}:=\{a_n\}_{n\in \Z}$ is stationary, so is $\textbf{N}:=\{N_n\}_{n\in \Z}$, assuming the population process starts at $-\infty$.

Consider a stationary integer-valued discrete sequence $\{U_n\}_{n\in \Z}$
in which $U_n$ equals $0$ or $1$. Let $\{k_n\}_{n\in \Z}$, with
\begin{align}
\label{ordering}
\cdots<k_{-1}<k_0\leq 0 < k_1 <k_2 <\cdots,
\end{align}
be the sequence of times at which $U_n=1$. A simple stationary point process (henceforth ${s.s.p.p.}$)
on $\Z$ is then a random counting measure $\Phi(\cdot)=\sum_{n\in \Z}\delta_{k_n}(\cdot)$, 
where $\delta_{k_n}(\cdot)$ is the Dirac measure at $k_n$.
Throughout the document, all ${s.s.p.p.}$s will be assumed to such that the associated $\Phi$ is a.s. not equal to the
empty measure.  
We often identify $\Phi$ with the sequence $\{k_n\}_{n\in \Z}$, writing $k_n\in \Phi$, whenever $\Phi(\{k_n\})=1$. 

The intensity of a ${s.s.p.p.}$, denoted by $\lambda_{\Phi}$, is given by $\PP[0\in \Phi]$.  

Let $\Phi$ be a ${s.s.p.p.}$ on $\Z$ such that $\lambda_{\Phi}>0$ and let $\PP_{\Phi}[\cdot]:=\PP[\cdot|0\in \Phi]$. Then $\PP_{\Phi}$ is the {\em Palm probability} of $\Phi$. We denote by $\E_{\Phi}$ the expectation operator of $\PP_{\Phi}$. Consider the operator $\theta_{k_1}$, i.e.,
 \begin{align}
 \label{invariantshift}
 \Phi\circ \theta_{k_1}:=\{k_{n+1}\}_{n\in \Z}.
 \end{align} 
 Since $\theta_{k_1}$ is a bijective map on $\Omega_0=\{k_0(\omega)=0\}$, the following holds \cite{heveling2005characterization}. 
\begin{lem}
\label{palmpreservation}
The operator $\theta_{k_1}$ preserves the Palm probability $\PP_{\Phi}$. 
\end{lem} 
 
A ${s.s.p.p.}$ $\Psi$ on $\Z$ is an {\em integer-valued renewal process} if $\{k_n-k_{n-1}\}_{n\in \Z}$ is $\hbox{i.i.d.}$ under $\PP_{\Psi}$. 

Finally, we say that a stationary process $\textbf{R}:=\{R_n\}_{n\in \Z}$ is {\em regenerative} if there exists an integer-valued renewal process $\Psi=\{k_n\}_{n\in \Z}$, such that, under $\PP_{\Psi}$, $(\{k_n-k_{n-1}\}_{n > j},\{R_n\}_{n\geq  k_j})$ is independent of $\{k_n-k_{n-1}\}_{n\leq j}$ for all $j\in \Z$ and its distribution does not depend on $j$. Moreover, $\textbf{R}$ is {\em regenerative with independent cycles} if $\textbf{R}$ is regenerative and $\{R_{n}\}_{n< 0}$ is independent of $\{R_n\}_{n\geq 0}$ under $\PP_{\Psi}$. We call $\{k_n\}_{n\in \Z}$ the regeneration points of $\textbf{R}$. 

For most of this work, unless otherwise stated, we assume: $\textbf{a}:=\{a_n\}_{n\in \Z}$ is $\hbox{i.i.d.}$, $\E[a_0]<\infty$, and $\PP[a_0=1]>0$.

\begin{rem}
\label{aismixing}
Since randomness in this model comes from the sequence $\{a_n\}_{n\in \Z}$ and $\theta_1$ preserves $\PP$, by assuming $\{a_n\}_{n\in \Z}$ is $\hbox{i.i.d.}$, there is no loss of generality in assuming $(\Omega,\mathcal{F},\PP,\{\theta_n\}_{n\in \Z})$ is strongly mixing, i.e., 
$$\lim_{n\to \infty} \PP[A\cap \theta_{\pm n} B]=\PP[A]\PP[B]~\forall~A,B\in \mathcal{F}.$$
\end{rem}

\subsection{General case: main results} 
\label{generalcasemainresults}
\begin{thm}
\label{originalancestorsthm}
Let $\Psi^o:=\{m\in \Z: N_m=1\}$. Then, $\Psi^o$ is an integer-valued renewal process with intensity $\lambda^o:=\lambda_{\Psi^o}=\prod_{i=1}^{\infty} \PP[a_0\le i]>0$.  
\end{thm}

The atoms of $\Psi^o$ are called \emph{original ancestors} as all individuals born after any of them are necessarily its descendants in the family tree $T^f$ studied in Section \ref{netree}. 

\begin{proof}[Proof of Theorem \ref{originalancestorsthm}]
First we prove $\PP[N_0=1]>0$. By (\ref{individualsalive}),
$\PP[N_0=1]=\PP[\cap_{j=1}^{\infty} a_{-j}\le j]$.
Then, as the sequence $\{a_n\}_{n\in \Z}$ is \iid, 
\begin{align*}
\PP[N_0=1]=\prod_{j=1}^{\infty}\PP[a_0\le j].
\end{align*}
Since $\PP[a_0=1]>0$, none of the elements of the above product equals $0$. Hence, we can take logs on both sides to get:
\begin{align*}
\ln(\PP[N_0=1])&=\sum_{j=1}^{\infty}\ln(1-\PP[a_0>j])\geq \sum_{j=1}^{\infty}\ln(1-\PP[a_0\ge j])\\
&\ge 
\sum_{j=1}^{j^*-1} 
\ln(1-\PP[a_0>j])
-2\sum_{j=j^*}^{\infty} \PP[a_0\ge j]=C-2\E[a_0]>-\infty, 	
\end{align*}
where $C= \sum_{j=1}^{j^*-1}  (\ln(1-\PP[a_0>j]) +2 \PP[a_0>j])$. 
Here we used the fact that $\PP[a_0\ge j]> \frac 1 2$ for finitely many $j$ to define $j^*$,
the first $j$ such that $\PP[a_0\ge j]\le \frac 1 2$, and
the fact that $-x\leq \ln\left(1-\frac{x}{2}\right)$ if $x\in [0,1]$.
Consequently, $\PP[N_0=1]>0$. Stationarity of $\textbf{N}$ implies, for all $n\in \Z$, $\PP[N_n=1]=\PP[N_0=1]>0$. 

Since $(\Omega,\mathcal{F},\PP,\{\theta_n\}_{n\in \Z})$ is strongly mixing (Remark \ref{aismixing}), it is ergodic. Therefore, by Birkhoff's pointwise ergodic theorem, for all measurable functions $g:\Omega\to \R^+$ such that $\E[g]<\infty$, 
\begin{align*}
\lim_{n\to \pm\infty}\frac{1}{n}\sum_{i=1}^{\pm n}g\circ \theta_{\pm n}=\E[g]~~\PP-\hbox{a.s.}.
\end{align*} 

Let $g=\textbf{1}\{N_0=1\}$. Then, 
\begin{align*}
\lim_{n\to \pm\infty}\frac{1}{n}\sum_{i=1}^{\pm n}\textbf{1}\{N_0=1\}\circ \theta_{\pm n}=\PP[N_0=1]>0~~\PP-\hbox{a.s.}.
\end{align*} 

Hence, there exists a subsequence of distinct integers $\Psi^o:=\{k^o_n\}_{n\in \Z}$, satisfying (\ref{ordering}) such that $N_{k^o_n}=1$ for all $n\in \Z$. That $\Psi^o$ is a renewal process is proved in Appendix \ref{someproofsa}, Proposition \ref{propositionregardingrenewal}.  
\end{proof} 

In order to show that $\textbf{N}$ is a stationary regenerative process with respect to $\Psi^o$ with independent cycles, we rely on the following lemma. 
\begin{lem}
\label{independencefrompastfuture}
Under $\PP_{\Psi^o}$, $\{a_n\}_{n<0}$ is independent of $\{a_n\}_{n\geq 0}$. 
\end{lem}

\begin{proof}
Let $g,f:\Omega\to \R^+$ be two measurable, continuous, and bounded functions. Using the fact that the event $\{k^o_0=0\}$ is equal, by definition, to $\{N_0=1\}=\{\cap_{i> 0} a_{-i}\leq i\}$, we have
\begin{align*}
\E_{\Psi^o}[g(\{a_n\}_{n< 0})f(\{a_n\}_{n\geq 0})]&=\E[g(\{a_n\}_{n< 0})f(\{a_n\}_{n\geq 0})|k^o_0=0]\\
&=\E[g(\{a_n\}_{n< 0})f(\{a_n\}_{n\geq 0})|\cap_{i> 0} a_{-i}\leq i].
\end{align*} 
Now as $\{a_n\}_{n\geq 0}$ is independent of $\{a_n\}_{n< 0}$ under $\PP$,
\begin{align*}
& \E[g(\{a_n\}_{n< 0})f(\{a_n\}_{n\geq 0})|\cap_{i> 0} a_{-i}\leq i]\\
&=\E[g(\{a_n\}_{n< 0})|\cap_{i> 0} a_{-i}\leq i]\E[f(\{a_n\}_{n\geq 0})|\cap_{i> 0} a_{-i}\leq i]\\
 &=\E_{\Psi^o}[g(\{a_n\}_{n< 0})]\E_{\Psi^o}[f(\{a_n\}_{n\geq 0})],
\end{align*}
completing the proof. 
\end{proof}

\begin{cor}
The population process $\textbf{N}$ is a stationary regenerative process with respect to $\Psi^o$ with independent cycles. 
\end{cor}

\begin{proof}
Given $k^o_0=0$ (i.e., under $\PP_{\Psi^o}$), $N_0=1$ is a constant. Then \break $(\{k^o_n-k^o_{n-1}\}_{n>0},\{N_n\}_{n\geq 0})$ is a function of $\{a_n\}_{n\geq 0}$, while $\{k^o_n\}_{n\leq 0},$ is a function of $\{a_n\}_{n<0}$. Then, by Lemma \ref{independencefrompastfuture},
$(\{k_n-k_{n-1}\}_{n > 0},\{N_n\}_{n\geq  k_j})$ is independent of $\{k_n-k_{n-1}\}_{n\leq j}$ for $j=0$. Following the same reasoning as in Lemma \ref{independencefrompastfuture},  $\{a_n\}_{n\geq k_j}$ is independent of $\{a_n\}_{n<k_j}$ for all $j\in \Z$. It follows that $(\{k_n-k_{n-1}\}_{n > 0},\{N_n\}_{n\geq  k_j})$ is independent of $\{k_n-k_{n-1}\}_{n\leq j}$ for all $j$. We conclude $\{N_n\}_{n\in \Z}$ is a regenerative process. 

Independence of the random vectors $\{a_n\}_{k_j<n\leq k_{j+1}}$, also a consequence of Lemma \ref{independencefrompastfuture}, implies $\{N_n\}_{n\in \Z}$ is regenerative with independent cycles.  
\end{proof} 

\begin{rem}
When we do not assume that $\PP[a_0=1]>0$, let $\underline{m}>1$ be the smallest integer such that $\PP[a_0=\underline{m}]>0$. Then, $\PP[N_0<\underline{m}]=0$, while $\PP[N_0=\underline{m}]=\prod_{j=\underline{m}}^{\infty}\PP^0[a_0\le j]>0$. Proceeding as in the proof of Theorem \ref{originalancestorsthm}, we conclude there exists an integer-valued renewal process, $\tilde{\Phi}$ such that $k_n\in \tilde{\Phi}$ if and only if $N_{k_n}=\underline{m}.$  
\end{rem}

\subsection{Analytical properties of $N_n$}
\label{analyticsofN}

We now turn to some analytical properties of $\textbf{N}$. Some of our results are adapted from the literature on GI/GI/$\infty$ queues, i.e., queues with an infinite number of servers, and independently distributed arrival times and service times. Proofs can be found in Appendix \ref{someproofsa}. The results in this subsection do not depend on the assumption $\PP[a_0=1]>0$.

The following proposition can be extended to the case in which $\{a_n\}_{n\in \Z}$ is stationary rather than $\hbox{i.i.d.}$.
\begin{prop}
\label{proppopulation1}
For all $n\in \Z$, $\E[a_0]=\E[N_n]$.  
\end{prop}   

Proposition \ref{proppopulation1} is the Little's law for the infinite server queue, i.e., the expected number of individuals being served equals the arrival rate times the expected service time given that there is an arrival at the origin a.s. (see, e.g. \cite{asmussen2003applied}).  This shows the mean number of customers in steady state is finite if and only if the expected service time is finite.

A general formula for the moment generating function of the number of customers in steady state in a GI/GI/$\infty$ queue can be found in \cite{yucesanrare}. We limit ourselves to showing  $N_0$ is a light-tailed random variable, regardless of the distribution of $a_0$. The latter property has the following intuitive basis: for the value of $N_0$ to be large, it is necessary that several realizations of $\{a_n\}_{n<1}$ are large as well. Another way to get an intuition for this result is if we let the arrival times to be exponentially distributed with parameter $\lambda$. Then one can show $N_0$ has a Poisson distribution with parameter $\lambda \E[a_0]$, and it is light-tailed regardless of the tail of $a_0$.

\begin{prop}
\label{propmgfMn}
The moment generating function of $N_n$ is given by 
\begin{align}
\label{mgfpopulationlookatit}
\E[e^{t N_0}]=e^{t}\prod_{i=1}^{\infty}\left(e^{t}\PP[a_0>i]+\PP[a_0\leq i]\right)~~\forall~t\in \R.
\end{align}
Moreover, $\E[e^{t N_0}]<\infty$ for all $t\in \R$.  
\end{prop}

\subsection{Geometric marks}
\label{geometricmarks}

In general, the process $\textbf{N}$ is not Markovian. However, it is when the marks are geometrically distributed. Let us assume $a_0$ is geometrically distributed supported on $\bbN_+$ with parameter $s$. Let $r=1-s$.  Then, due to the memoryless property, $\textbf{N}$ is a time-homogeneous, aperiodic, and irreducible Markov chain with state space $\{1,2,3,\ldots\}$, whose transition matrix is given by
\begin{align}
\label{populationmarkovchain}
	\PP[N_n=n|N_{n-1}=k]&=\left(\begin{array}{c} k \\ n-1\end{array}\right) r^{n-1} s^{k-n+1},~1\leq n\leq k+1.
\end{align}

\begin{prop}
In the geometric case, the probability generating function of $N_n$ at steady-state obeys the following functional relation 
\begin{align}
\label{mgfpopulation}
	G(z)=zG(sz+r).
\end{align}
\end{prop}

\begin{proof}
Let $\{\tilde{N}_n\}_{n\geq 0}$ the population process starting at $0$ and $G_m(z)$, $z\in [0,1]$, the probability generating function of $\tilde{N}_m$. Using (\ref{populationmarkovchain}),

\begin{align*}
\mathbb{E}[z^{\tilde{N}_n}|\tilde{N}_{n-1}=k]&=\sum_{n=1}^{k+1}\left(\begin{array}{c} k \\ n-1\end{array}\right) r^{n-1} s^{k-n+1}z^n=z\sum_{n'=0}^{k}\left(\begin{array}{c} k \\ n'\end{array}\right) r^{n'} s^{k-n'}z^{n'},
\end{align*}
where $n'=n-1$. 
So, 
\begin{align*}
\mathbb{E}[z^{\tilde{N}_n}|\tilde{N}_{n-1}=k]&=z(rz+s)^k.
\end{align*}
Hence, 
\begin{align*}
	G_m(z)&=\sum_{k=1}^{m} \PP[\tilde{N}_{m-1}=k](z(rz+s)^k)=zG_{m-1}(rz+s).
\end{align*}
Letting $m\to \infty$ we get (\ref{mgfpopulation}). 

\end{proof} 

Using (\ref{mgfpopulation}) we can easily compute the moments of $N_0$. For example, by differentiating both sides and setting $z=1$, we get
\begin{align}
	\E[N_0]&=\frac{1}{s},
\end{align}
which is the mean of $a_0$, as expected. Proceeding the same way, the second moment is given by
\begin{align}
	\E[N_0^2]&=\frac{2r}{s(1-r^2)}.
\end{align}

\begin{rem}
By (\ref{mgfpopulationlookatit}), and noticing $\PP[a_i>i]=r^i$, we get
\begin{align}
G(z)&=z\prod_{i=1}^{\infty}(zr^i+1-r^i).
\end{align}

By Lemma \ref{populatioauxlem1} in Appendix \ref{someproofsa}, $\prod_{i=1}^{\infty}(zr^i+1-r^i)$ converges for all $z\in \R$ as $\sum_{i=1}^{\infty} r^i<\infty$. 
 \end{rem} 

\section{The Eternal Family Tree} 
\label{netree}

In this section we study the directed graph $T^f=(V^f,G^f)$, where $V^f=\Z$ and $E^f=\{(n,f(n)):n\in \Z\}$. In Section \ref{globalpropertiesoftf}, we show $T^f$ is an infinite tree containing a unique bi-infinite path. The indexes of the nodes on this bi-infinite path form a ${s.s.p.p.}$ on $\Z$ with positive intensity. We derive this result by exploiting the fact that $T^f$ is an \textit{Eternal Family Tree}, i.e., the out-degrees of all vertices are exactly one \cite{BOAunimodular}.  
 
In the following two sections, we delve deeper into the genealogy of $T^f$. 
In Section \ref{genealogyofTf}, we give the basic properties of certain ${s.s.p.p.}$s derived from $T^f$. First, the process of integers forming the bi-infinite path: each integer in it is called \emph{successful}, since its lineage (the set of its descendants) has infinite cardinality a.s. Second, we consider the process coming from the complement of the bi-infinite path: each integer in it is called  {\em ephemeral}, since its lineage is finite a.s. Third, we consider the process of original ancestors, defined in Theorem \ref{originalancestorsthm}, which is a subprocess of the first.  

In Section \ref{directephemeralsubsubsection}, we look at the set of direct ephemeral descendants of a successful integer $n$, whose path to $n$ on $T^f$ consists only of ephemerals.  The conservation law that defines unimodular networks allows us to establish probabilistic properties of the set of direct ephemerals descendants and the set of cousins of a typical successful node. 

\subsection{The global properties of $T^f$} 
\label{globalpropertiesoftf}

\begin{thm}
\label{genealogythm}
The directed random graph $T^f$ is a tree with a unique bi-infinite path for which the corresponding nodes, when mapped to $\Z$, form a ${s.s.p.p.}$ with positive intensity.
\end{thm}

In order to prove Theorem \ref{genealogythm}, we resort to recent results on dynamics on unimodular networks \cite{BOAunimodular}. In Appendix \ref{Unimodulardynamics}, we present a brief review of definition and properties of unimodular networks that we use. First, we notice the directed graph $G=(V,E)$, where $V=\Z$ and $E=\{(n,n+1):n\in \Z\}$ in which each node $n$ has mark $a_n$, rooted at $0$, is a locally finite unimodular random network. Second, we notice $f$ is a translation-invariant dynamics on this network, more precisely, a \emph{covariant vertex-shift} (see Appendix \ref{Unimodulardynamics}, Definition \ref{covariantvertexshift}).  

Define the \emph{connected component} of $n$ as 
\begin{align}
\label{connectedcomponent}
C(n)=\{m\in \Z~\hbox{s.t.}~\exists~i,j\in \bbN~\hbox{with}~f^i(n)=f^j(m)\}.
\end{align} 

\begin{prop}
\label{uniquecomponentproposition}
The directed random graph $T^f$ has only one connected component. 
\end{prop}
\begin{proof}
Consider the process of original ancestors, $\Psi^o$, as defined in Theorem \ref{originalancestorsthm}, and let $m\in \Psi^o$. Then, for every $n<m$, $n\in C(m)$. Hence, the result follows from the fact that $\Psi^o$ is a ${s.s.p.p.}$ consisting of an infinite number of points $\hbox{a.s.}$
\end{proof}

Let 
\begin{align}
\label{descendants}
D(n)=\{m\in \Z~\hbox{s.t.}~\exists~j\in \bbN~\hbox{with}~f^i(m)=n\}	
\end{align} 
denote the set of {\em descendants} of $n$. Also let, 
\begin{align}
\label{foilofn}
L(n)=\{m\in \Z~\hbox{s.t.}~\exists~j\in \bbN~\hbox{with}~f^j(m)=f^j(n)\}	
\end{align}
denote the set of {\em cousins} of $n$ of all degrees (this set is referred to as the {\em foil} of $n$ in \cite{BOAunimodular}). 

We further subdivide $D(n)$ and $L(n)$ in terms
\begin{align}
D_i(n)=\{m\in \Z~\hbox{s.t.}~f^i(m)=n\},~i\geq 0
\end{align}
and
\begin{align}
\label{foilofol}
L_i(n)=\{m\in \Z~\hbox{s.t.}~f^i(m)=f^i(n)\},~i\geq 0.
\end{align}
So $D_i(n)$ is the set of descendants of degree $i$ of $n$ and $L_i(n)$ the set of cousins of degree $i$ of $n$. 

Lower case letters denote the cardinalities of the above sets. So $c(n)$ is the cardinality of $C(n)$, $d_i(n)$ is the cardinality of $D_i(n)$ and so on. Moreover, $d_{\infty}(n)$ denotes the weakly limit of $d_i(n)$ (if such a limit exists).

In \cite{BOAunimodular}, it is shown that each connected component  of a graph  generated by the action of a covariant vertex-shift on a unimodular network, $C(n)$, falls within one of the following three categories: 
\begin{enumerate}
\item Class \textbf{F}/\textbf{F}: $c(n)<\infty$ and for all $v\in C(n)$, $l(v)<\infty$. In this case $C(n)$ has a unique cycle.   
\item Class \textbf{I}/\textbf{F}: $c(n)=\infty$ and for all $v\in C(n)$, $l(v)<\infty$. In this case, $C(n)$ is a tree containing a unique bi-infinite path. Moreover, the bi-infinite path forms a ${s.s.p.p.}$ on $\Z$ with positive intensity. 
\item Class \textbf{I}/\textbf{I}: $c(n)=\infty$ and for all $v\in C(n)$, $l(v)=\infty$. In this case, $C(n)$ is also a tree such that $d_{\infty}(v)=0$ for all $v\in C(n)$. 
\end{enumerate}
Notice the dynamics $f$ precludes the connected component $\Z$ of being of class \textbf{F}/\textbf{F}. 
Theorem \ref{genealogythm} follows from proving that, in our case, $C(0)$ is of class \textbf{I}/\textbf{F}.  We rely on the following lemma derived from the results found in  \cite{BOAunimodular}. 

\begin{lem}
\label{ClassificationLemma}
A connected component $C(m)$ is of class \textbf{I}/\textbf{I} if and only if for all $n\in C(m)$, $\PP[d(n)=\infty]=0$. 
\end{lem}

\begin{proof}[Proof of Theorem \ref{genealogythm}]
For $k\in \Psi^o$, we have $d(k)=\infty$ a.s., and consequently, $C(k)$ is of class \textbf{I}/\textbf{F}. Since there is a unique component, the result follows. 
\end{proof} 

\begin{rem}
Coming back to the case where it is not assumed that $\PP[a_0=1]>0$, let $\underline{m}>1$ be the smallest integer such that $\PP[a_0=\underline{m}]>0$. 
Define the sets
\begin{align*}
Y(i)&=\{m\in \Z:~\PP[m\in C(i)]>0\}~\hbox{for $i\in \{0,\ldots,\underline{m}-1\}$},
\end{align*}
and notice $\{Y(i)\}_{i=0}^{\underline{m}-1}$ forms a partition of $\Z$. Let
\begin{align*}
N^i_m=\#\{m<n, m\in Y(i): f(m)>n\}+1.
\end{align*}
Then, using the same arguments as in the proof of Theorem \ref{originalancestorsthm}, one can show that, for each $i\in \{0,\ldots,\underline{m}-1\}$, there exists a ${s.s.p.p}$, $\Psi_o^i$, such that $k \in \Psi_i^o$ if and only if $N^i_k=1$. It follows that $Y(i)=C(i)$.  Hence, $T^f$ has $\underline{m}$ connected components. By translation invariance, all connected components are of class \textbf{I}/\textbf{F}. In this case, $T^f$ is a forest. 
\end{rem}

\begin{defi}[Diagonally invariant functions]
\label{translationinvariant} 
A measurable function $g: \Omega \times \Z\times \Z\to \R$ is said to be diagonally invariant if 
$h(\theta_k(\omega),m,n)=h(\omega,m+k,+k)$ for all $k,m,n\in \Z$.
\end{defi}

Finally, $T^f$ itself is a unimodular network (\cite{BOAunimodular}). Unimodularity is characterized by the fact such a network obeys {\em the mass transport principle} (see Appendix \ref{Unimodulardynamics} and \cite{BOAunimodular}). In our setting, the mass transport principle takes the following form, recalling $0$ is by convention the root of $T^f$: for all diagonally invariant functions $h$ (Definition \ref{translationinvariant}),
\begin{align}
\label{themtp}
\E\left[\sum_{n\in \Z} h(n,0)\right]=\E\left[\sum_{n\in \Z} h(0,n)\right]. 
\end{align} 

\subsection{Successful and ephemeral individuals, and original ancestors}
\label{genealogyofTf}

From the analysis of the population process and the shape of $T^f$, we learned $f$ defines three ${s.s.p.p.}$s on $\Z$ related to the genealogy of $T^f$:
\begin{enumerate}
\item $\Phi^s$: the set of successful individuals, consisting of individuals $n\in \Z$ having an infinite number of descendants in $T^f$.
\item $\Psi^o$: the set of original ancestors (defined in Section \ref{populationdynamics}), consisting of all individuals $n\in \Z$ such that for all $m<n$,  $m$ is a descendant of $n$ in $T^f$. Clearly, $\Psi^o\subset \Phi^s$. 
\item $\Phi^e$: the set of ephemeral individuals, consisting of individuals $n\in \Z$ which have a finite number of descendants in $T^f$. Clearly, $\Phi^e\cup \Phi^s=\Z$. 
\end{enumerate}
We now look at the basic properties of these processes. In what follows we let $\E^s:=\E_{\Phi^s}$, i.e., $\E^s$ is the expectation operator of the Palm probability of $\Phi^s$. In the same vein, $\E^e:=\E_{\Phi^e}$, and $\E^o:=\E_{\Psi^o}$.

\begin{prop}
Let $\lambda^s$ be the intensity of $\Phi^s$. Then, 
\begin{align}
\label{lambdas}
\lambda^s=\frac{1}{\E[a_0]}.
\end{align}
It follows that the intensity of $\Phi^e$, $\lambda^e$, equals $1-\frac{1}{\E[a_0]}$. 
\end{prop}

\begin{proof}
Let $S_0=0$ and $S_n=a_1+a_2+\cdots+a_n$.  The associated renewal sequence $\{u_k\}_{k\geq 0}$ is defined as:
\begin{align*}
u_k=\PP[S_n=k~\hbox{for some $n\geq 0$}],	
\end{align*}
so $u_k$ is the probability that $k$ is hit at some renewal epoch $S_n$. 
Then from asymptotic renewal theory (see, e.g., \cite{asmussen2003applied}), since the distribution of $\{a_n\}_{n\in \Z}$ is aperiodic (as $\PP[a_0=1]>0$), $u_k\to \frac{1}{\E[a_0]}$ as $k\to \infty$, $\PP-$a.s. As $\lim_{k\to \infty} u_k=\PP[0\in \Psi^s]=\lambda^o$ and $\lambda^s+\lambda^e=1$, the result follows.  
\end{proof}

\begin{rem}
As $\prod_{i=1}^{\infty}\PP[a_0\leq i]\le \frac{1}{\E[a_0]}$, $\lambda^o\le \lambda^s$, as expected. 
\end{rem}

We know from \cite{BOAunimodular} that $\E[d_n(0)]=1$, i.e., the expected number of descendants of all degrees of a typical integer is one. This result follows from the mass transport principle. The process of successful individuals is locally supercritical, while the process of ephemeral individuals is locally subcritical, as the next proposition shows. 

\begin{prop}
Assume $\PP[a_0=1]\in (0,1)$. Then, for all $n\geq 1$, $\E^s[d_n(0)]>1$, while $\E^e[d_n(0)]<1$. 
\end{prop}

\begin{proof}
By the law of total probability and the definition of $\PP^s$ and $\PP^e$, 
\begin{align*}
\E^s[d_n(0)]\lambda^s+\E^e[d_n(0)]\lambda^e=\E[d_n(0)]=1.
\end{align*}
Since a successful has at least one descendant of degree $n$ $\hbox{a.s.}$, $\PP[a_0=1]<1$, and $\lambda^e=1-\lambda^s$,
$\E^s[d_n(0)]>1$ and  $\E^e[d_n(0)]<1$. 
\end{proof}

\subsection{Cousins and direct ephemeral descendants}
\label{directephemeralsubsubsection}

Given a successful node $n$, we say that an ephemeral individual $m$  is a {\em direct ephemeral descendant} of $n$ if $n$ is the first successful in the ancestry lineage of $m$. The set of direct ephemeral descendants of $n$ is hence:
\begin{align*}
D^e(n)=\{m\in D(n)\cap \Phi^e:\hbox{for the smallest $k>0$ s.t. $f^k(m)\in \Phi^s$, $f^k(m)=n$}\}, 
\end{align*} 
where $D(n)$ is the set of descendants of all degrees of $n$ (see Equation (\ref{descendants})).  By Theorem \ref{genealogythm}, the cardinality of $D^e(n)$, denoted by $d^e(n)$, is finite. Moreover, for $m\neq n\in \Psi^s$, $D^e(n)\cap D^e(m)=\emptyset$.

\begin{defi}[Direct ephemeral descendants partition] 
Let $$\tilde{D}^e(n):=D^e(n)\cup \{n\}$$ be the directed ephemeral tree rooted at $n\in \Phi^s$.
The {\em direct ephemeral descendant partition} is \begin{align}\mathcal{P}^D:=\{\tilde{D}^e(n):{n\in \Phi^s}\}.\end{align} 
\end{defi}

We notice that, by convention, any individual $n$ is a cousin of itself (i.e., $n$ is a $0-$degree cousin of itself). Moreover, if $m\neq n$ both belong to $\Psi^s$, then $L(n)\cap L(m)=\emptyset$, as either $m$ is a descendant or an ancestor of $n$. In other words, for $n\in \Phi^s$, $L(n)\backslash \{n\}\subset \Phi^e$. Hence, we get the following partition.  

\begin{defi}[Successful cousin partition]
The cousin partition of $\Z$ is \begin{align}\mathcal{P}^L:=\{L(n):{n\in \Phi^s}\}.\end{align}
\end{defi}

\begin{figure}
\center
      \includegraphics[width=1\textwidth]{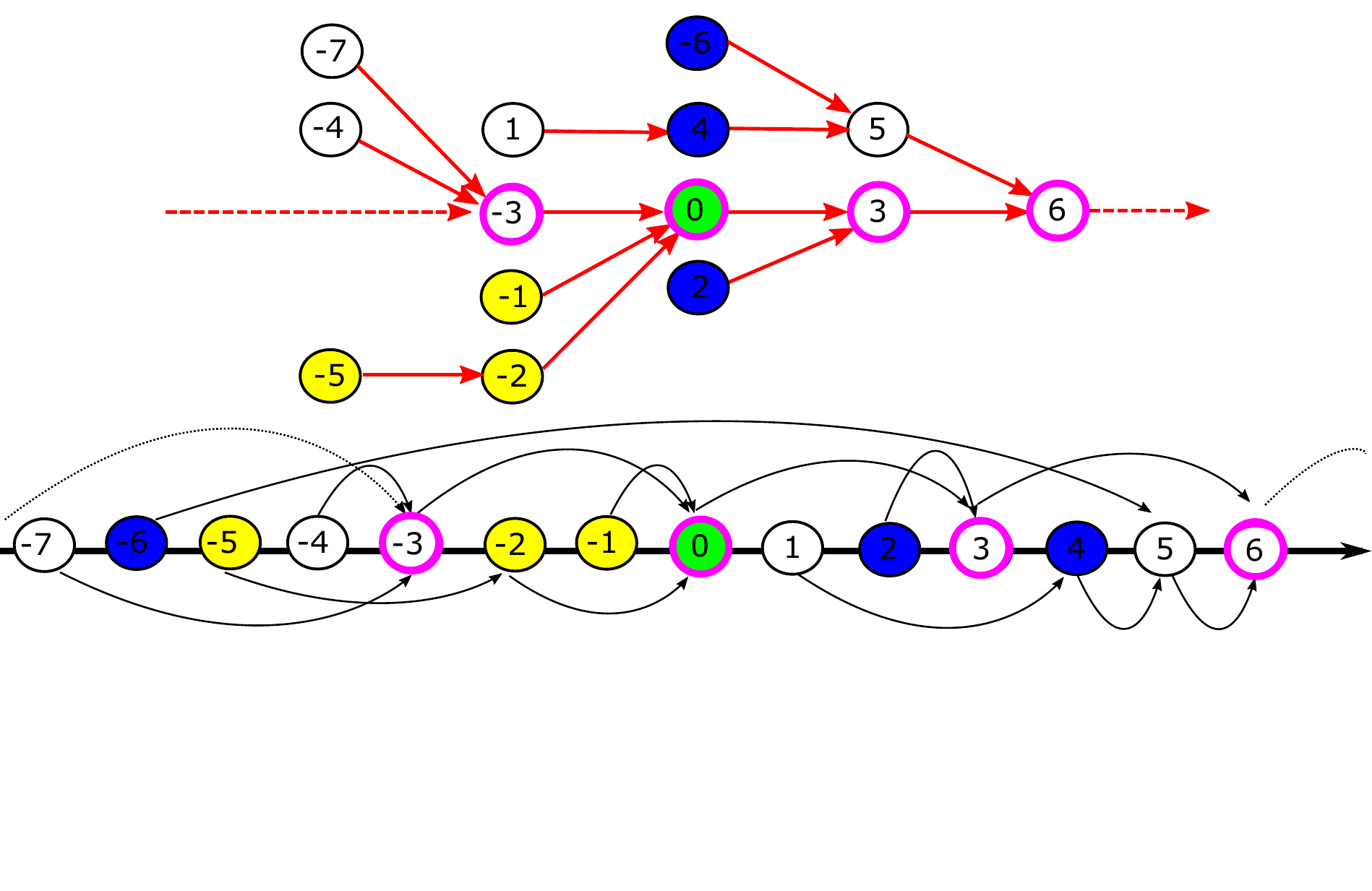}
  \caption{\textbf{The direct ephemeral descendants and the cousins of 0.} Above we have a realization of $T^f$ and below the corresponding integers on $\Z$. Here, -3, 0, 3, and 6  belong to the bi-infinite path (denoted by circles with pink boundaries). The yellow individuals are the direct ephemeral descendants of $0$, while the blue ones are cousins. Individual $0$ has two ephemeral children (nodes -1 and -2), one successful (node -3), and one ephemeral grandchild (node -5). It has a first degree cousin (node 2) and two second degree cousins (nodes -6 and 4). While any descendant of $0$ must be to the left of it on $\Z$, cousins can be either to the left or right. }
\end{figure}

Figure 3 illustrates $\tilde{D}^e(0)$ and $L(0)$.  For all $n\in \Psi^e$ and $j>0$, let $d^e_j(n)=\#\{D^e(n)\cap D_j(n)\}$ be the number of directed ephemeral descendants of degree $j$ of $n$. By construction, the following equality holds for all $j>0$ $\PP^s-\hbox{a.s.}$ (see Figure 4):
\begin{align}
\label{cousintreerelationship1}
l_j(0)&=d^e_j(k^s_j),
\end{align}
so that
\begin{align}
\label{cousintreerelationship2}
l(0)&=\sum_{j=1}^{\infty}d^e_j(k^s_j)+1.
\end{align}

\begin{figure}
\center
      \includegraphics[width=0.7\textwidth]{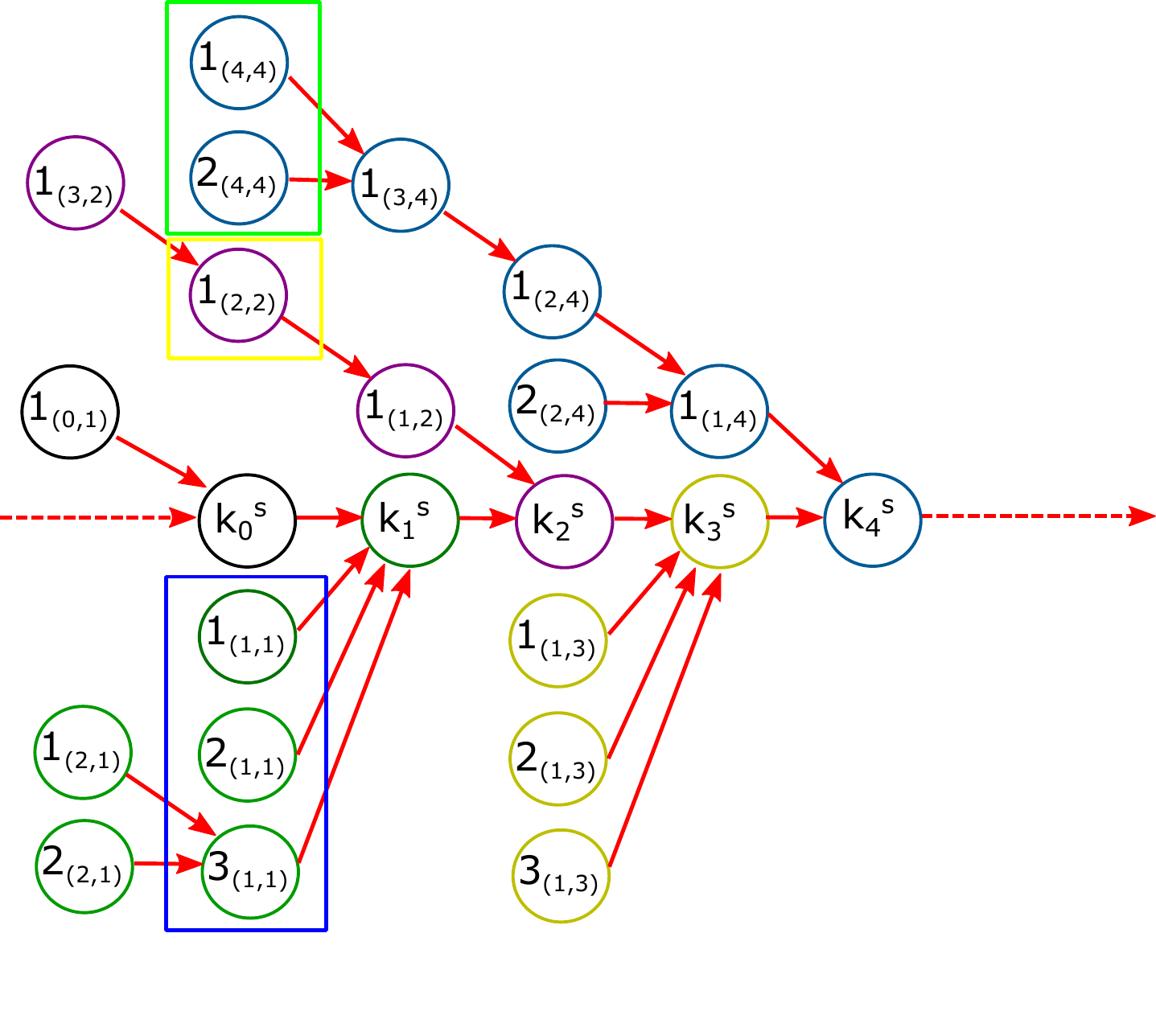}
  \caption{\textbf{Cousins and the direct ephemeral descendants trees:} The integers $\{k^s_i\}_{i=1}^{4}$ represent the successful individuals. The color of the boundary of the circles indicates the direct ephemeral descendant tree an individual belongs to. For example, all individuals represented by blue boundary circles belong to the direct ephemeral descendant tree of $k^s_4$. The notation $I_{(n,m)}$ reads ``individual $I$ of the $n^{th}$ layer of the direct ephemeral descendant tree of $k^s_m$''. For example, $2_{(2,4)}$ is the second individual in the second layer of the direct ephemeral descendant tree of $k^s_4$. Each colored box contains the cousins of $k^s_0$ of different degrees. The blue box contains the first-degree, the yellow box contains the second-degree, and the green box contains the forth-degree cousins (there are no third-degree cousins). Equivalently, the blue box contains all elements of the first layer of the direct ephemeral descendant tree of $k^s_1$, the yellow box contains all elements of the second layer of the direct ephemeral descendant tree of $k^s_2$, and the green box contains all elements of the forth layer of the direct ephemeral descendant tree of $k^s_4$. As the direct ephemeral descendants tree of $k^s_3$ has no third layer, $k^s_0$ has no third-degree cousins. }
\end{figure}

\begin{prop}
\label{immortalprop1}
 For all $j\geq 1$ and $q\ge 0$, $\PP^s[d^e_j(0)=q]=\PP^s[l_j(0)=q]$.  
\end{prop}
\begin{proof}
From (\ref{cousintreerelationship1}), for all $j\geq 1$ and $q\ge 0$,
\begin{align}
\label{immortalprop1eq1}
\PP^s[l_j(0)=q]=\PP^s[d^e_j(k^s_j)=q]. 
\end{align} 
As $\theta_{k^s_j}$ preserves $\PP^s$,
\begin{align}
\label{immortalprop1eq2}
\PP^s[\theta_{k^s_j}\{d^e_j(k^s_j)=q\}]&=\PP^s[d^e_j(0)=q],\quad\forall~q\geq 1. 
\end{align}
We get the result by combining Equations (\ref{immortalprop1eq1}) and (\ref{immortalprop1eq2}).
\end{proof}

\begin{prop}
\label{masstpdirectdescendantsandfoil}
Given $0\in \Phi^e$, let $n^{(d)}$ be the unique random successful individual such that $0\in \tilde{D}^e(n^{(d)})$. In the same way, let $n^{(l)}$ be the unique successful individual such that $0\in L(n^{(l)})$. Then, for any diagonally invariant function $g$ (Definition \ref{translationinvariant}),
\begin{align}
\label{mtpdirectdescendants}
\lambda^s\E^s\left[\sum_{n\in \tilde{D}^e(0)}g(0,n)\right]&=\lambda^s\E^s[g(0,0)]+\lambda^e\E^e[g(n^{(d)},0)]
\end{align}
and
\begin{align}
\label{mtpfoildirectdescendants}
\lambda^s\E^s\left[\sum_{n\in L(0)}g(0,n)\right]&=\lambda^s\E^s[g(0,0)]+\lambda^e\E^e[g(n^{(l)},0)].
\end{align}
\end{prop}

\begin{proof}
Equation (\ref{mtpdirectdescendants}) follows from applying the mass transport principle to the function
\begin{align*}
h_1(\omega,0,n):=\textbf{1}\{0\in \Phi^s(\omega)\}\textbf{1}\{n\in \tilde{D}^e(0)(\omega)\}g(\omega,0,n),
\end{align*}
while (\ref{mtpfoildirectdescendants}) follows from applying the mass transport principle to the function
\begin{align*}
h_2(\omega,0,n):=\textbf{1}\{0\in \Phi^s(\omega)\}\textbf{1}\{n\in L(0)(\omega)\}g(\omega,0,n).
\end{align*}
\end{proof}

\begin{cor}
\label{cormtp}
The following holds:
\begin{align}
\label{mtpsamemean}
\E^s[\tilde{d}^e(0)]=\E^s[l(0)]=\E[a_0],
\end{align} where $\tilde{d}^e(0)$ is the cardinality of $\tilde{D}^e(0)$. Moreover,
\begin{align}
\label{mtpmean2eq}
\frac{\E^s\left[\sum_{n\in \tilde{D}^e(0)}|n|\right]}{\E^e[|n^{(d)}|]}=\frac{\E^s\left[\sum_{n\in L(0)}|n|\right]}{\E^e[|n^{(l)}|]}=\E[a_0]-1.
\end{align}
\end{cor}

\begin{proof}
The results follow from choosing particular $g(0,n)$ in Proposition \ref{masstpdirectdescendantsandfoil}.
Let $g(0,n)\equiv 1$. Then, Equation (\ref{mtpsamemean}) holds as $\lambda^s+\lambda^e=1$ and $\lambda^s=\frac{1}{\E[a_0]}.$
\item Set $g(0,n)=|n|$. Again, using Equation (\ref{mtpdirectdescendants})
\begin{align*}\
\lambda^s\E^s\left[\sum_{n\in \tilde{D}^e(0)}|n|\right]&=\lambda^s\E^s[0]+\lambda^e\E^e[|n^{(d)}|]=\lambda^e\E^e[|n^{(d)}|].
\end{align*}
Hence, 
\begin{align*}
\E^e[|n|^{(d)}]&=\left(\frac{\E[a_0]}{\E[a_0]-1}\right)\frac{1}{\E[a_0]}\E^s\left[\sum_{n\in \tilde{D}^e(0)}|n|\right]\\
&=\frac{\E^s\left[\sum_{n\in \tilde{D}^e(0)}|n|\right]}{\E[a_0]-1}.
\end{align*}
Following the same steps using (\ref{mtpfoildirectdescendants}), we recover Equation (\ref{mtpmean2eq}). 
\end{proof} 

\section{Final remarks}

Here are a few comments on the population dynamics interpretation
of the model discussed here. Our model is concerned with a critical 
population dynamics. In branching processes, the critical case leads to extinction unless there is no variability at all. In contrast, in our model, there is no extinction, although the population is infinitely often close to extinction, as the original ancestor point 
process shows. In connection with this, we find it interesting to mention that there 
is some genetic and archaeological evidence that the human population was close to extinction several times 
in the distant past. (\cite{article2} and \cite{article1}).

\appendix

\section{Proofs: Section \ref{populationdynamics}}
\label{someproofsa}

\begin{prop}
\label{propositionregardingrenewal}
The ${s.s.p.p.}$ of original ancestors constructed in Theorem \ref{originalancestorsthm} is a renewal process. 
\end{prop}
\begin{proof}
Notice that, for all $m$, 
\begin{align*}
\PP_{\Psi^o}[(k^o_1-k^o_0)=m]=\PP[k^o_1=m|k^o_0=0]=:y_m.
\end{align*}
Now, for any $n\in \Z$, as $\theta_{k_n}$ preserves $\PP_{\Psi^o}$,
\begin{align*}
\PP_{\Psi^o}[(k^o_{n+1}-k^o_n)=m]&=\PP_{\Psi^o}[(k^o_{n+1}-k^o_n)\circ \theta_{k^o_n}=m]\\
&=\PP_{\Psi^o}[(k^o_1-k^o_0)=m]=y_m.
\end{align*}
Hence $\{k^o_{n+1}-k^o_n\}_{n\in \Z}$ is identically distributed. 

Now let $i_0<i_1<i_2<\cdots<i_q$ be a finite set of integers. Then, for $m_0,\ldots,m_q\in \bbN$, 
\begin{align*}
&\PP_{\Psi^o}[(k^o_{i_0+1}-k^o_{i_0})=m_0,(k^o_{i_2+1}-k^o_{i_2})=m_1,\ldots,(k^o_{i_q+1}-k^o_{i_q})=m_q]\\
&=\PP_{\Psi^o}[(k^o_{i_0-i_q+1}-k^o_{i_0-i_q})=m_0,(k^o_{i_1-i_q+1}-k^o_{i_1-i_q})=m_1,\ldots,k^o_{1}=m_q].
\end{align*}
Now, under $\PP_{\Psi^o}$ the realization of $k^o_{1}$ is a function of $\{a_n\}_{n\geq 0}$, and for all $i<0$, $(k^o_{i}-k^o_{i-1})$ is a function of $\{a_n\}_{n<0}$. Hence, as $\{a_n\}_{n\geq 0}$ is independent of $\{a_n\}_{n<0}$ (Lemma \ref{independencefrompastfuture}), 
\begin{align*}
&\PP_{\Psi^o}[(k^o_{i_0-i_q+1}-k^o_{i_0-i_q})=m_0,(k^o_{i_1-i_q+1}-k^o_{i_1-i_q})=m_1,\ldots,k^o_{1}=m_q]\\
&=\PP_{\Psi^o}\left[(k^o_{i_0-i_q+1}-k^o_{i_0-i_q})=m_0,(k^o_{i_1-i_q+1}-k^o_{i_1-i_q})=m_1,\ldots,\right.\\
&\left.(k^o_{i_{q-1}-i_q+1}-k^o_{i_{q-1}-i_q})=m_{q-1}\right]\PP_{\Psi^o}[k^o_{1}=m_q]\\
&=\PP_{\Psi^o}[(k^o_{i_0+1}-k^o_{i_0})=m_0,(k^o_{i_2+1}-k^o_{i_2})=m_1,\ldots,(k^o_{i_{q-1}+1}-k^o_{i_{q-1}})=m_{q-1}]\\
&\times \PP_{\Psi^o}[(k^o_{i_q+1}-k^o_{i_q})=m_q].
\end{align*}
Keep proceeding in the same way, we conclude that
\begin{align*}
&\PP_{\Psi^o}[(k^o_{i_0+1}-k^o_{i_0})=m_0,(k^o_{i_2+1}-k^o_{i_2})=m_1,\ldots,(k^o_{i_q+1}-k^o_{i_q})=m_q]\\
&=\prod_{j=0}^q\PP_{\Psi^o}[(k^o_{i_j+1}-k^o_{i_j})=m_j].
\end{align*}
Therefore, $\{k^o_{n+1}-k^o_n\}_{n\in \Z}$ is also an independent sequence. 
\end{proof}

\begin{proof}[Proof of Proposition \ref{proppopulation1}]
\label{proofofmomentproposition}
Let $Z_{n,m}=\textbf{1}\{f(m)>n\}$. Then, notice that
\begin{align}
\label{populationequation1}
N_n=\sum_{m<n}^{\infty}Z_{n,m} +1.
\end{align}
For $m\leq 0$, let $M_m=\sum_{m\leq l<0}^{\infty}Z_{0,l}+1$,  so $\lim_{m\to -\infty} M_m=N_0$ $\hbox{a.s.}$ Then, using $a_m\overset{(d)}{=}a_0$ for all $m\in \Z$, 
\begin{align*}
\E[M_m]&=\sum_{m\leq l<0}\PP[a_{-l}>l]+1\\
&= \sum_{l=1}^m\PP[a_{0}> l]+\sum_{l=1}^{m}\PP[a_0=l]+\sum_{l>m}\PP[a_0=l]\\
&= \sum_{l=1}^m\PP[a_{0}\geq l]+\sum_{l>m}\PP[a_0=l].
\end{align*}
Letting $m\to \infty$, by monotone convergence, we get $\E[N_0]= \E[a_0]$. \end{proof} 

In order to prove Propostion \ref{propmgfMn}, we use the following results.
\begin{lem}
\label{populatioauxlem1}
Consider the infinite product $\prod_{i=1}^{\infty}(1+b_i)$, where $b_i\geq 0$ for all $i$. Then if $\sum_{i=1}^{\infty} b_i$ converges (resp. diverges), then $\prod_{i=1}^{\infty}(1+b_i)$ also converges (resp. diverges). 
\end{lem}
For a proof, see e.g. \cite{heinbockel2010introduction}.

\begin{proof}[Proof of Proposition \ref{propmgfMn}]

Using (\ref{populationequation1}), we can compute the moment generating function of $N_0$ as follows. For $t\in \R$:
\begin{align}
\label{populationequation2}
\E[e^{t N_0}]&=\E[e^{t}e^{t \sum_{i=1}^{\infty}Z_{0,i}}]=e^{t}\prod_{i=1}^{\infty}\E[e^{t Z_{0,i}}]\nonumber\\
&=e^{t}\prod_{i=1}^{\infty}(e^{t}\PP[a_0>i]+\PP[a_0\leq i]).
\end{align}
Thus, $
\E[e^{t N_0}]=e^{t}\prod_{i=1}^{\infty}(b_i+1),
$
where $b_i=\PP[a_0>i](e^t-1)$. As $\E[a_0]<\infty$, $\sum_{i=1}^{\infty}b_i<\infty$ and, therefore, by Lemma \ref{populatioauxlem1}, $\E[e^{t N_0}]<\infty$. 
\end{proof} 

\section{Dynamics on Unimodular Networks}
\label{Unimodulardynamics}  

We review the necessary concepts on unimodular networks dynamics used in this paper. We borrow the concepts of this section from \cite{BOAunimodular} and \cite{aldous2007processes}, which contain a more complete treatment of the subject. 

\begin{defi}[Locally finite rooted networks]	A network is quadruple \break $(G,\Xi,u_1,u_2)$ where:
\begin{enumerate}
	\item $G=(V,E)$ is a (multi-)graph;
	\item $\Xi$ is a complete separable metric space;
	\item $u_1:V\to \Xi$. The element $u_1(v)$ is called the mark of the vertex $v$;
	\item $u_2:\{(v,e):v\in V,e\in E,v\sim e\}\to \Xi$. The element $u_2(v,e)$ is called the mark of the pair $(v,e)$. 
\end{enumerate}
If $G$ has a distinguished vertex, we call the network rooted.  A network is locally finite if all its vertices have finite degrees. We also consider the case in which $G$ has two distinguished vertices. 
\end{defi}

A isomorphism between graphs $G$ and $G'$ is a bijection that preserves the (direct) vertices. A isomorphism between rooted  networks also preserves its marks and maps the distinguished vertex (or vertices) of $G$ to $G'$.    

Let $\mathcal{G}_*$ (resp. $\mathcal{G}_{**}$) be the space isomorphism classes  of locally finite rooted (resp. with two distinct vertices) networks and $\mathcal{B}(\mathcal{G}_*)$ (resp. $\mathcal{B}(\mathcal{G}_{**}))$ its Borel-$\sigma$ algebra. An element of $\mathcal{G}_*$ (resp. $\mathcal{G}_{**})$ is denoted by $[G,o]$ (resp. $[G,o,o']$). We denote the set of vertices of any representant of the isomorphism class $[G,o]$ (or $[G,o,o']$) by $V$. 

A random rooted (resp. with two distinct vertices) network is measurable mapping from a general probability space $(\Omega,\mathcal{F}, \PP)$ to $(\mathcal{G}_*,\mathcal{B}(\mathcal{G}_*))$ (resp. $(\mathcal{G}_{**},\mathcal{B}(\mathcal{G}_*))$). We denote a random network by $[\textbf{G},\textbf{o}]$ (resp. $[\textbf{G},\textbf{o},\textbf{o'}]$). 

A locally finite network is unimodular if for all measurable functions $g:\mathcal{G}_{**}\to \mathbb{R}^{\geq}:$
\begin{align*}
\e\left[\sum_{v\in V}g[\textbf{G},\textbf{o},v]\right]=\e\left[\sum_{v\in V}g[\textbf{G},v,\textbf{o}]\right].	
\end{align*}
This definition does not depend  of the choice of
representative of $[\textbf{G},\textbf{o},\textbf{v}]$.

\begin{defi}
\label{covariantvertexshift}
A \textbf{covariant vertex-shift} is a map $X_G$ which associates to each unimodular network a function $x_G:V\to V$ such that
\begin{enumerate}
\item $x_G$ is covariant under isomorphisms and
\item the function $[G,o,o']\to \ind\{x_G(o)=o'\}$ is measurable on $\calG_{**}$. 	
\end{enumerate}
\end{defi}

From vertex-shift $X_G$ acting on $[G,o,o']$ we let $G^{X}=(V^X,E^X)$ be the graph such that $V^X=V$ and $E^X=\{v,x_G(v)\}_{v\in V}$. We then define the following subsets of $V^X$. 

\begin{defi}[Connected Component]
	The \emph{connected component} of $v$ under the action of vertex-shift $x_G$ is given by 
	\begin{align}
	\label{connectedcomponent2}
	C(v):=\{w\in V~\hbox{s.t.}~\exists~i,j\in \bbN~\hbox{with}~x_{G}^i(v)=x^j_G(w)\}.
	\end{align}
\end{defi}

\begin{defi}[Foil]
\label{foil} 
The \emph{foil} of $v$ is defined as
	\begin{align}
	L(v)=\{w\in V~\hbox{s.t.}~\exists~j\in \bbN~\hbox{with}~x_G^j(w)=x^j_G(v)\}.	
	\end{align}
\end{defi}

We let $c(v)$ (resp. $l(v)$) denote the cardinality of $C(v)$ (resp. $L(v)$). In~\cite{BOAunimodular}, it is showed each connected component of $G^X$, in particular $C(o)$, falls within one of the following categories. 

\begin{enumerate}
\item Class \textbf{F}/\textbf{F}: $c(o)<\infty$ and for all $v\in C(o)$, $l(v)<\infty$. 
\item Class \textbf{I}/\textbf{F}: $c(o)=\infty$ and for all $v\in C(o)$, $l(v)<\infty$. 
\item Class \textbf{I}/\textbf{I}: $c(o)=\infty$ and for all $v\in C(o)$, $l(v)=\infty$.
\end{enumerate}

\section*{Acknowledgments}
This work was supported by a grant of the Simons Foundations (\#197982 to
the University of Texas at Austin). The authors would like to thank James Murphy III for valuable comments.
\bibliographystyle{amsplain}
\bibliography{nearextinction} 
\end{document}